\documentclass[12pt]{amsart}

\usepackage{amsmath,amsthm,amsfonts,amscd,amssymb,amsopn,enumerate}
\usepackage{eucal}
\usepackage{mathrsfs,color,hyperref}
\usepackage{wasysym}
\usepackage[all]{xy}
\usepackage{fullpage}

\newcommand{\mbb}[1]{\mathbb #1}

\newcommand{\ms}[1]{\mathscr #1}

\newcommand{\wh}{\widehat}
\newcommand{\Spec}{\operatorname{Spec}}
\newcommand{\GL}{\operatorname{GL}}
\newcommand{\PGL}{\operatorname{PGL}}

\newcommand{\Hom}{\operatorname{Hom}}

\newcommand{\cd}{\operatorname{cd}}
\newcommand{\red}{{\operatorname{red}}}
\newcommand{\Pic}{{\operatorname{Pic}}}
\newcommand{\Cl}{{\operatorname{Cl}}}

\newcommand{\Frac}{{\operatorname{frac}}}

\newcommand{\til}[1]{\widetilde{#1}}

\newcommand{\customdiagram}[4]{
\fbox{\xymatrix @C=.33cm @R=-0.3cm {
 & {#2} \ar@/^0.4pc/[dr] \\
{#1} \ar@/^0.4pc/[ur] \ar@/_0.4pc/[dr] & & {#4} \\
 & {#3} \ar@/_0.4pc/[ur]
}}
}

\newcommand{\vardiagram}[2]{
\fbox{\xymatrix @C=.33cm @R=-0.3cm {
 & {#2}_1 \ar@/^0.4pc/[dr] \\
{#2}_0 \ar@/^0.4pc/[ur] \ar@/_0.4pc/[dr] & & {#1} \\
 & {#2}_2 \ar@/_0.4pc/[ur]
}}
}

\newcommand{\mvardiagram}[2]{
\mbox{\xymatrix @C=.33cm @R=-0.3cm {
 & {#2}_1 \ar@/^0.4pc/[dr] \\
{#2}_0 \ar@/^0.4pc/[ur] \ar@/_0.4pc/[dr] & & {#1} \\
 & {#2}_2 \ar@/_0.4pc/[ur]
}}
}

\newcommand{\varPdiagram}[2]{
\fbox{\xymatrix @C=.33cm @R=-0.3cm {
 & {#2}_\UU \ar@/^0.4pc/[dr] \\
{#2}_\BB \ar@/^0.4pc/[ur] \ar@/_0.4pc/[dr] & & {#1} \\
 & {#2}_\PP \ar@/_0.4pc/[ur]
}}
}

\newcommand{\antidiagram}[1]{
\fbox{\xymatrix @C=.33cm @R=-0.3cm {
 & {#1}_1 \ar@/_0.4pc/[dl] \\
{#1}_0 & & {#1} \ar@/_0.4pc/[ul] \ar@/^0.4pc/[dl] \\
 & {#1}_2 \ar@/^0.4pc/[ul]
}}
}

\theoremstyle{plain}
\newtheorem{thm}{Theorem}[section]

\newtheorem{lem}[thm]{Lemma}

\newtheorem{cor}[thm]{Corollary}
\newtheorem{prop}[thm]{Proposition}

\newtheorem*{thm*}{Theorem}
\newtheorem*{rem*}{Remark}
\newtheorem*{lem*}{Lemma}

\newtheorem*{cor*}{Corollary}
\newtheorem*{prop*}{Proposition}

\theoremstyle{definition}

\newtheorem{ex}[thm]{Example}

\theoremstyle{remark}
\newtheorem{rem}[thm]{Remark}

\newcommand{\SL}{\operatorname{SL}}

\newcommand{\sheaf}[1]{\mathscr{#1}}
\newcommand{\XX}{\sheaf{X}}
\newcommand{\PP}{\sheaf{P}}
\newcommand{\UU}{\sheaf{U}}
\newcommand{\BB}{\sheaf{B}}
\newcommand{\OO}{\sheaf{O}}
\newcommand{\KK}{\sheaf{K}}

\renewcommand{\P}{\mathbb P}

\newcommand{\QQ}{\mathbb{Q}}

\newcommand{\A}{\mathbb{A}}

\newcommand{\Z}{\mathbb{Z}}

\newcommand{\Gm}{\mathbb{G}_{\mathrm m}}

\newcommand{\fM}{\mathfrak M}
\newcommand{\h}{\mathrm h}

\newcommand{\oper}[1]{\operatorname{#1}}

\newcommand{\cha}{\oper{char}}

\usepackage[OT2,T1]{fontenc}

\renewcommand{\labelenumi}{(\theenumi)}
\def\<{\left<}
\def\>{\right>}

\newcommand{\Br}{\oper{Br}}
\newcommand{\Orth}{\oper{O}}

\renewcommand{\theenumi}{\alph{enumi}}
\renewcommand{\labelenumi}{(\alph{enumi})}

\DeclareSymbolFont{cyrletters}{OT2}{wncyr}{m}{n}
\DeclareMathSymbol{\Sha}{\mathalpha}{cyrletters}{"58}

\title{Adelic double cosets over semi-global fields}
\date{\today}
\author{ David Harbater}
\thanks{
The author was supported in part by NSF grant DMS-2102987.}
\begin{document}

\maketitle

\begin{abstract}
Double coset spaces of adelic points on linear algebraic groups arise in the study of global fields; e.g., concerning local-global principles and torsors.  A different type of double coset space plays a role in the study of semi-global fields such as $p$-adic function fields.  This paper relates the two, by establishing adelic double coset spaces over semi-global fields; relating them to local-global principles and torsors; and providing explicit examples.
\end{abstract}

\section{Introduction}

Adelic double coset spaces play an important role in the study of linear algebraic groups $G$ over a global field $F$.  In work beginning with Borel (see \cite{Borel}), these spaces have been used in understanding local-global principles, class groups (using $G=\Gm$ or more generally $\GL_n$), and the genus of quadratic forms (using $G = \Orth(n)$), as well as other applications (see \cite[Section~1.2]{Con}, \cite{CRR}).   For any
finite set $S$ of places of $F$ including all the archimedean places (if any), the double coset space
$G(F) \backslash G(\A(F,S)) / G(\A^\infty(F,S))$
is an abelian group if $G$ is commutative, and in general is a pointed set.  Because of the relationship to the class group of a number field or the Picard group of a curve over a finite field, this space is also referred to as the ``class set'' or ``class group'' of the group $G$ (see \cite[Chapter~8]{PlaRap}, \cite[Section~2]{RR:CR}).

A different type of double coset space has arisen in the study of semi-global fields such as $\QQ_p(x)$; i.e., one-variable function fields $F$ over a complete discretely valued field $K$.  It involves only finitely many factors (unlike adelic double cosets), and provides the obstruction to a local-global principle for $G$-torsors over $F$ in the context of patching over fields (see \cite{HHK:H1}).  This patching obstruction approximates a more canonical local-global obstruction $\Sha_X(F,G)$, which is taken with respect to points on the closed fiber $X$ of a projective model $\XX$ of $F$ over the valuation ring $\OO_K$; and that in turn can be used in studying the local-global obstruction $\Sha(F,G)$ with respect to discrete valuations on $F$ (see \cite[Proposition~8.2]{HHK:H1}, and \cite[Theorem~3.2]{CHHKPS}).

This paper introduces an adelic double coset space in the context of semi-global fields that makes it possible to study $\Sha_X(F,G)$ more directly, rather than via approximation by patching obstructions.  It also provides a clearer link between the global and semi-global contexts.  
Via results that have already been proven about $\Sha_X(F,G)$ and its patching analog, this double coset space can be computed in a number of situations.  Its definition, however, is not quite what one might have expected, due to the fact that the intersection property for patching (asserting that $F$ is the intersection of the finitely many overfields in the double coset space; see \cite[Section~2]{HHK:H1}) turns out not to hold in the adelic situation.  But the modified adelic double coset space introduced below nevertheless agrees with $\Sha_X(F,G)$, as a consequence of Artin's Approximation Theorem \cite{artin}; see Corollary~\ref{double coset H1}.  

The above obstructions are contained in the Galois cohomology set $H^1(F,G)$; and if $G$ is commutative then this is a group and higher Galois cohomology groups $H^i(F,G)$ can also be considered.  In that situation, we define higher double coset spaces that can be directly computed using higher obstructions $\Sha_X^i(F,G) \subset H^i(F,G)$; see Corollary~\ref{double coset Hi}.  

Whereas these double coset spaces are defined in terms of fields, closer analogs to the global field case, involving local rings, are introduced afterwards, in the context of one-dimensional schemes (generalizing the global field situation) and for projective models of semi-global fields.  These analogs agree with the obstructions to local-global principles for torsors over such schemes in the context of \'etale cohomology.  (See Corollaries~\ref{double coset H1 curve} and \ref{double coset H1 rings}.)  
In the case of $G=\Gm$, this double coset space agrees with the Picard group of the scheme,  and so generalizes the notion of the class set (or group) of $G$ in the theory of global fields.
(For a different construction in a somewhat similar spirit, see \cite[Appendix~A]{CRR}, which generalizes a result of Harder, \cite[Section~2.3]{Harder}.)

This paper is organized as follows.
Following a discussion of preliminaries in Section~\ref{sect: background}, double coset spaces for semi-global fields are introduced in Section~\ref{sect: sgf Cl}, by building on results in the patching context.  In that context, the agreement of the finite double coset space with the local-global obstruction had been proven using a six-term Mayer-Vietoris sequence (see \cite[Theorem~3.5]{HHK:H1}); here a limiting argument yields a version of that sequence that relates $\Sha_X(F,G)$ to adelic double cosets.  For groups $G$ that are commutative, we can similarly rely on a long exact Mayer-Vietoris sequence that was given in the patching context in \cite{HHK:Hi}.  Afterwards,  Section~\ref{sect: ring Cl} establishes an adelic double coset space in terms of local rings, which is thus closer to the situation in the global field case.  This is done for reasonable one-dimensional schemes and for models of semi-global fields, using a Tannakian approach to patching that was recently established in \cite{HKL}.  In each of these situations, the double coset space is interpreted as a local-global obstruction, and examples are given that compute the double coset spaces explicitly.

I thank Julia Hartmann for comments on this manuscript.

\section{Background on semi-global fields and patching} \label{sect: background}

This section reviews the set-up established in \cite{HH:FP}, \cite{HHK}, and \cite{HHK:H1}.
Let $T$ be a complete discrete valuation ring, with fraction field $K$, residue field $k$, and uniformizer $t$.
A {\it semi-global field} $F$ {\it over} $K$ (or {\it over} $T$) is a finitely generated field over $K$ of transcendence degree one in which $K$ is algebraically closed.  A (projective) {\it model} $\XX$ of $F$ is an integral flat projective $T$-curve with function field $F$.  For example, $\P^1_{\Z_p}$, and a blow up of $\P^1_{\Z_p}$ at a closed point, are both models of $\QQ_p(x)$.  We may in particular consider models that are normal or regular as schemes.  A regular model whose closed fiber $X := \XX \times_T k$ is a normal crossings divisor (i.e., each irreducible component is regular and the components meet at ordinary double points) is called a {\it normal crossings model}.  (See \cite[Chapter~10]{Liu} for more about such models.)

We may consider local-global principles for torsors under a linear algebraic group $G$ over a semi-global field $F$.  Here and below, we require all linear algebraic groups to be smooth as schemes.  Under that assumption, the Galois cohomology set $H^1(F,G)$ classifies isomorphism classes of $G$-torsors over $F$ (as in the global field case), and 
we may consider
\[\Sha(F,G) := \ker(H^1(F,G) \to \prod_v H^1(F_v,G)),\]
where $v$ runs over the equivalence classes of non-archimedean absolute values on $F$; or equivalently, here, the discrete valuations on $F$.  This Tate-Shafarevich set is the obstruction to a $G$-torsor over $F$ being trivial (or equivalently, having an $F$-point) whenever it is trivial over each $F_v$.  (Some authors have restricted $v$ just to those discrete valuations that are trivial on $K$, in which case the obstruction is in general larger; e.g., see 
\cite{CH}, \cite{HSS}, \cite{HS16}.)

A variant local-global obstruction for a semi-global field $F$, also motivated by the global field case, concerns the closed fiber $X$ of a model $\XX$ of $F$, and is given by
\[\Sha_X(F,G) := \ker(H^1(F,G) \to \prod_{P\in X} H^1(F_P,G)),\]
where $P$ ranges over all the points of $X$ (including generic points) and where $F_P$ is the fraction field of the complete local ring $\wh R_P := \wh \OO_{\XX,P}$.  

These two obstructions are pointed sets, and are abelian groups if $G$ is commutative.  By \cite[Proposition~8.2]{HHK:H1}, $\Sha_X(F,G) \subseteq \Sha(F,G)$ if $\XX$ is a regular model, and so the local-global principle with respect to discrete valuations can be shown to fail by proving the non-vanishing of $\Sha_X(F,G)$.  Moreover, in some situations it is known that $\Sha_X(F,G) = \Sha(F,G)$; e.g., see \cite[Theorem~3.2]{CHHKPS}).  

To study $\Sha_X(F,G)$, one can approximate it by smaller obstruction sets $\Sha_\PP(F,G)$, where $\PP$ ranges over non-empty finite sets of closed points of $X$ that contain all the singular points of $X$.  These are defined using rings and fields that arise in the framework of patching, where they had earlier been used in inverse Galois theory over semi-global fields (e.g., \cite{Ha87}, \cite{Pr}, \cite{Ha:MSRI}).  These rings and fields are motivated by the analogy between viewing the general fiber of $\XX \to \Spec(T)$ as a $K$-curve under the $t$-adic topology (as in rigid geometry), and viewing a complex algebraic curve as a Riemann surface under the complex metric topology.  

More precisely, if $U$ is a non-empty irreducible affine open subset of $X$, let $R_U \subset F$ be the subring of elements regular along $U$; and let $\wh R_U$ be the $t$-adic completion of $R_U$, with fraction field $F_U$.  
Thus $\wh R_U \subseteq \wh R_{U'}$ and $F_U \subseteq F_{U'}$ if $U \subseteq U'$.
If $P$ is a closed point of $X$, then a {\it branch} of $X$ at $P \in X$ is a height one prime $\wp$ of $\wh R_P$ containing $t$.  These are in bijection with the irreducible components of the restriction of $X$ to $\Spec(\wh R_P)$.  For such $P$ and $\wp$, let $R_\wp$ be the localization of $\wh R_P$ at $\wp$; and let $\wh R_\wp$ be the completion of $R_\wp$, with fraction field $F_\wp$.  Now given a choice of a finite set $\PP$ as above, let $\UU$ be the set of connected components of $X \smallsetminus \PP$, and let $\BB$ be the set of branches of $X$ at the points of $\PP$.
We thus have a finite set of overfields of $F$; viz., $F_P$ (for $P \in \PP$), $F_U$ (for $U \in \UU$), and $F_\wp$ (for $\wp \in \BB$).  Here $F_\wp$ contains $F_P$ and $F_U$, if $\wp$ is a branch at $P$ on (the closure of) $U \in \UU$; it also contains $F_\eta$, where $\eta$ is the generic point of $U$.  We may view each $F_P$ and $F_U$ as the field of meromorphic functions on a $t$-adic neighborhood on the general fiber of $\XX$, with the fields $F_\wp$ corresponding to the overlaps.  Analogously to the behavior of meromorphic functions with respect to a metric open covering of a Riemann surface, the set of fields $F_P$ and $F_U$ satisfies an {\it intersection property}: their intersection 
(or strictly speaking, inverse limit) is equal to $F$,  viewing $F_P,F_U \subset F_\wp$ as above.

The above fields $F_P,F_U,F_\wp$ have been useful in computing invariants of semi-global fields, especially concerning quadratic forms and Brauer groups (e.g., see \cite{HHK}, \cite{PS}).  They also give rise to 
the local-global obstruction
\[\Sha_\PP(F,G) := \ker\bigl(H^1(F,G) \to \prod_{\xi \in \PP \cup \UU} H^1(F_\xi,G)\bigr).\]  
If $\XX$ is a normal model of $F$, then $\Sha_\PP(F,G) \subseteq \Sha_X(F,G)$ for each choice of $\PP$ as above; moreover $\Sha_X(F,G)$ is the union of its subsets $\Sha_\PP(F,G)$, if $\PP$ ranges over all finite sets as above (see \cite[Corollary~5.9]{HHK:H1}).  

One can identify $\Sha_\PP(F,G)$ with a (finite) double coset space via a six-term Mayer-Vietoris sequence in Galois cohomology.  Namely, by~\cite[Theorem~3.5]{HHK:H1}, there is an exact sequence of pointed sets
\begin{equation} \label{ses diagram patch}
\small{
\xymatrix{
1 \ar[r] & H^0(F,G) \ar[r]^-{\alpha^0} &
\prod_{P \in \PP} H^0(F_P,G) \times \prod_{U \in \UU} H^0(F_U,G) \ar[r]^-{\beta^0} &
\prod_{\wp \in \BB} H^0(F_\wp,G) \ar@<-2pt> `d[l]
`[lld]^-{\delta^0} [lld] \\
& H^1(F, G) \ar[r]^-{\alpha^1} & \prod_{P \in \PP} H^1(F_P,G) \times \prod_{U \in \UU}
H^1(F_U,G) \ar@<.5ex>[r]^-{\beta_1^1} \ar@<-.5ex>[r]_-{\beta_2^1} & \prod_{\wp \in \BB}
H^1(F_\wp,G). 
\\
}}
\end{equation} 
Here $H^0(F,G) = G(F)$, and similarly for the other terms in the first row.  Thus we obtain a bijection of pointed sets 
\begin{equation} \label{sgp bij}
\prod_{U \in \UU} G(F_U) \ \backslash {\prod_{\wp \in \BB}} \ G(F_\wp) \ / \prod_{P \in \PP} G(F_P) \cong \Sha_\PP(F,G),
\end{equation} 
between a double coset space and a Tate-Shafarevich set, in the patching context.  A key ingredient is the intersection property.

In the case of a commutative group $A$, the Galois cohomology sets $H^i(F,A)$ are abelian groups and are defined for all $i \ge 0$.  Under additional hypotheses (that $\cha(k)=0$ or $A$ has finite order not divisible by $\cha(k)$), it was shown in \cite[Theorem~3.1.2]{HHK:Hi} that
(\ref{ses diagram patch}) extends to a long exact Mayer-Vietoris sequence:
\begin{equation} \label{les diagram patch}
\small{
\xymatrix{
0 \ar[r] & H^0(F,A) \ar[r]^-\Delta & \prod_{P \in \PP} H^0(F_P,A) \times \prod_{U \in
\UU} H^0(F_U,A) \ar[r]^-{-} & \prod_{\wp \in \BB} H^0(F_\wp,A) 
\ar `d[l] `[lld] [lld] \\
& H^1(F, A) \ar[r]^-\Delta & \prod_{P \in \PP} H^1(F_P, A) \times \prod_{U \in
\UU} H^1(F_U, A) \ar[r]^-{-} & \prod_{\wp \in \BB} H^1(F_\wp, A) 
\ar `d[l] `[lld] [lld] \\ & H^2(F, A) 
\ar[r]^-\Delta &
{\ \ \cdots .\hskip 2.2in}
}}
\end{equation}
This provides information about higher local global principles and higher double coset spaces, in the patching context.

\section{Adelic-type double cosets over semi-global fields} \label{sect: sgf Cl}

This section introduces an adelic-type double coset space in the context of a semi-global field, forming a bridge between the classical adelic double cosets for global fields and the finite double cosets over semi-global fields reviewed in Section~\ref{sect: background}.  This double coset space is then interpreted in terms of local-global principles via a Mayer-Vietoris sequence, and explicit examples are computed.  (Afterwards, Section~\ref{sect: ring Cl} considers a closer analog to the global field case, using local rings rather than fields.)

The double coset space we introduce here will be in bijection with the elements of the Tate-Shafarevich obstruction set $\Sha_X(F,G) := \ker(H^1(F,G) \to \prod_P H^1(F_P,G))$ that was recalled in Section~\ref{sect: background}.  It is thus natural to expect that it would be obtained by modifying the bijection (\ref{sgp bij}) by taking the products over $G(F_\wp)$ for all branches $\wp$ of the closed fiber $X$ of $\XX$ at closed points of $X$; over $G(F_P)$ for all closed points $P \in X$; and over $G(F_\eta)$ for all generic points $\eta$ of $X$ (one for each irreducible component).  But an unexpected issue arises; viz., the intersection property fails for these fields $F_P,F_\eta,F_\wp$:

\begin{ex} \label{int counterex}
Let $k$ be a field of characteristic zero; let $T=k[[t]]$; and let $\XX$ be the projective $x$-line over $T$.  
Then there is one generic point $\eta$ on the closed fiber $X = \P^1_k$.  The corresponding function field is the semi-global field $F=k((t))(x)$.
For each closed point $P$ of $X$, the field $F_\eta = k(x)((t))$ is contained in $F_\wp$, where $\wp$ is the branch of $X$ at $P$.  In particular, if $P$ is the point $x=c$ on $X$ for some $c \in k$, then $\wh R_P = k[[x-c,t]]$, $\wh R_\wp = k((x-c))[[t]]$, and 
$F_\wp = k((x-c))((t)) \supset F_\eta$.  Consider the element $a_\eta = \sum_{n=1}^\infty t^n/(x-n) \in F_\eta$.  In the case that $P$ is the point $x=m$ on $X$ for some positive integer $m$, the series $(x-m)\sum_{n=1}^\infty t^n/(x-n)$ is an element of in $k[[x-m,t]] = \wh R_P$; whereas for all other closed points $P \in X$,  the series $\sum_{n=1}^\infty t^n/(x-n)$ lies in $\wh R_P$. So in both cases, the element $a_\eta$, viewed as an element of $F_\wp$, is the image of an element $a_P \in F_P = \Frac(\wh R_P)$ given by the same expression.  
The elements $a_P$ and $a_\eta$ together define an element in the inverse limit of the fields $F_P$, $F_\eta$, and $F_\wp$ that is not induced by any element of $F$ (or even any element of some $F_U$, for $U$ an affine open subset of $X$), since it has infinitely many poles on the closed fiber $X$.  Thus this inverse limit is strictly bigger than $F$.  (See also Remark~\ref{6-term_pts_rk} below.)
\end{ex}

Instead, the double coset space we consider, and the associated Mayer-Vietoris sequence, will be somewhat different from the naive expectation.  To obtain it from the patching versions of those objects, we use the following lemma (where the ``kernel'' of a doubled map is the equalizer).

\begin{lem} \label{expand ex seq}
Let 
\begin{equation} \label{first ex seq}
D^0 \ {\buildrel \alpha^0 \over \to}\  A^0  \ {\buildrel \beta^0 \over \to}\  B^0 \ {\buildrel \delta^0 \over \to} \ D^1 \ {\buildrel \alpha^1 \over \to} \ A^1  \ 
\mathrel{\mathop{\rightrightarrows}^{\beta^1_1}_{\mathrm{\beta^1_2}}}\  B^1
\end{equation}
be an exact sequence of pointed sets.  Let $C^0$ be a group, let $C^1$ be a pointed set, and for $i=0,1$ let $\epsilon^i:A^i \to C^i$ be a map of pointed sets.
\renewcommand{\theenumi}{\alph{enumi}}
\renewcommand{\labelenumi}{(\alph{enumi})}
\begin{enumerate}
\item  \label{enlarged seq sets}
Then 
\begin{equation} \label{longer ex seq}
D^0 \ {\buildrel \alpha'^0 \over \to}\  A^0 \times C^0 \ {\buildrel \beta'^0 \over \to}\  B^0 \times C^0\ {\buildrel \delta'^0 \over \to} \ D^1 \ {\buildrel \alpha'^1 \over \to} \ A^1 \times C^1 \ 
\mathrel{\mathop{\rightrightarrows}^{\beta'^1_1}_{\mathrm{\beta'^1_2}}}\  B^1 \times C^1
\end{equation}
is also an exact sequence of pointed sets, where 
$\alpha'^i(d) = (\alpha^i(d),\epsilon^i \alpha^i(d))$ for $i=0,1$; where 
$\beta'^0(a,c) = (\beta^0(a),\epsilon^0(a)^{-1}\cdot c)$,
$\beta'^1_1(a,c)= (\beta_1^1(a),c)$,  and $\beta'^1_2(a,c)= (\beta_2^1(a),\epsilon^1(a)))$; and where $\delta'^0(b,c) = \delta^0(b)$.
\item \label{enlarged seq groups}
Suppose that $A^i,B^i,C^i,D^i$ are abelian groups; that $\alpha^i$, $\beta^0$, $\beta_j^1$, $\delta^0$, $\epsilon^i$ are group homomorphisms;
and that (\ref{first ex seq}) is an exact sequence of abelian groups if we replace the double arrow with the homomorphism $\beta^1$ given by $\beta^1(a)=\beta_2^1(a)^{-1}\beta_1^1(a)$. 
Then (\ref{longer ex seq}) is also an exact sequence of abelian groups if we similarly replace the double arrow with the map $\beta'^1$ given by $\beta'^1(a,c)=\beta'^1_2(a,c)^{-1}\beta'^1_1(a,c)$. 
\end{enumerate}
\end{lem}

\begin{proof}
In (\ref{enlarged seq sets}), the fact that (\ref{longer ex seq}) is a complex follows immediately from the definitions of the maps and the fact that (\ref{first ex seq}) is a complex.  For exactness at $B^0 \times C^0$, suppose that $\delta'^0(b,c)$ is the distinguished element of $D^1$.  Then $b=\beta^0(a)$ for some $a \in A^0$, by the exactness of (\ref{first ex seq}); and so $(b,c) = \beta'^0(a,c_0)$, where $c_0 = \epsilon^0(a)\cdot c \in C^0$.  Exactness at the other terms of (\ref{longer ex seq}) is immediate from the definitions of the maps.

In (\ref{enlarged seq groups}), the maps $\alpha'^i, \delta'^0$ are clearly homomorphisms; and so are the maps 
$\beta'^i$ because $(g,h) \mapsto h^{-1}g$ is a group homomorphism in the case of abelian groups.  Since a sequence of abelian groups is exact if and only if it is exact as a sequence of pointed sets, the assertion is immediate.
\end{proof}

We return to our geometric situation.  Let $\XX$ be a model of a semi-global field $F$ over $T$, with closed fiber $X$ over the residue field $k$ of $T$.  
For $i=0,1$, let $X_{(i)}$ denote the set of points of $X$ whose closure has dimension~$i$.  Thus 
$X_{(0)}$ is the set of closed points $P$ of $X$, and $X_{(1)}$
is the finite set of generic points of irreducible components of $X$.  For any subset $S \subseteq X$, let $\BB_S$ be the set of all branches of $X$ at closed points of $S$.  For each $P \in X_{(0)}$ we have the associated complete local ring $\wh R_P$ and its fraction field $F_P$; and similarly we have $\wh R_\eta$ and $F_\eta$ for each $\eta \in X_{(1)}$, and $\wh R_\wp$ and $F_\wp$ for $\wp \in \BB_X$.

For any functor $\Phi$ from field extensions of $F$ to (pointed) sets, define the 
\textit{restricted product} 
$\prod'_{\wp \in \BB_X} \Phi(F_\wp)$ to be the 
direct limit of the family 
$\{\prod_{\wp \in \BB_S} \Phi(F_\wp) \times \prod_{P \not\in S} \Phi(F_P)\}$,
indexed by finite subsets $S \subset X$.  
The natural map $\prod'_{\wp \in \BB_X} \Phi(F_\wp) \to \prod_{\wp \in \BB_X} \Phi(F_\wp)$ is injective provided that the maps $\Phi(F_P) \to \Phi(F_\wp)$ are injective for $\wp$ a branch at $P$.  
For example, for a linear algebraic group $G$ over $F$, there is the functor $H^i(-,G)$ for $i=0,1$, and also for $i>1$ if $G$ is commutative.  Here $H^0(F_P,G) \to H^0(F_\wp,G)$ is injective.  Hence
the restricted product $\prod'_{\wp \in \BB_X} G(F_\wp)$ is the subset of $\prod_{\wp \in \BB_X} G(F_\wp)$ consisting of the families $(g_\wp)_{\wp \in \BB_X}$ that satisfy $g_\wp \in G(F_P)$ for all but finitely many pairs $\wp,P$ where $\wp$ is a branch at $P$.  (But caution: $\prod'_{\wp \in \BB_X} H^1(F_\wp,G)\to\prod_{\wp \in \BB_X} H^1(F_\wp,G)$ is not injective; see Remark~\ref{6-term_pts_rk}.)

Similarly, if $G$ is a smooth affine group scheme of finite type over $\XX$, we can define the \textit{restricted product} 
$\prod'_{\wp \in \BB_X} H^i(\wh R_\wp,G)$ for $i=0,1$, and also for $i>1$ if $G$ is commutative.  Namely, this is the direct limit of the family $\{\prod_{\wp \in \BB_S} H^i(\wh R_\wp,G) \times \prod_{P \not\in S} H^i(\wh R_P,G)\}$,
indexed by finite subsets $S \subset X$.

The  restricted product, as opposed to the full product, is one of the two ingredients that we need to introduce in order to obtain a well-behaved double coset space over a semi-global field.  The other such ingredient is a field that will take the place of $F_\eta$, and which is described next. 

For $\eta \in X_{(1)}$, let $\UU_\eta$ be the collection of connected affine open neighborhoods of $\eta$ in $X$ all of whose closed points are unibranched on $X$; and let 
$R_\eta^\h \subseteq \wh R_\eta$ be the subring $\bigcup_{U \in \UU_\eta} \wh R_U$.  This is a Henselian discrete valuation ring with respect to the $\eta$-adic valuation, such that its fraction field $F_\eta^\h \subseteq F_\eta$ is the subfield $\bigcup_{U \in \UU_\eta} F_U$, and its residue field is the residue field at $\eta$ (i.e., the function field of the corresponding irreducible component of $X^\red$); 
see \cite[Lemma~3.2.1]{HHK:Hi}.  
Thus $R_\eta^\h$ contains the henselization $\til R_\eta$ of the local ring $R_\eta$; and indeed, every element of $\til R_\eta$ lies in $\wh R_U$ for some $U\in\UU_\eta$ (see the proof of \cite[Proposition~5.8]{HHK:H1}) and hence lies in~$R_\eta^\h$.

\smallskip

 Note that if a functor $\Phi$ as above is of finite presentation (i.e.\ if $\displaystyle\Phi(\lim_\to E_\nu) = \lim_\to(\Phi(E_\nu))$ for every direct system of $F$-algebras $\{E_\nu\}$), then there is a natural map
$\prod_{\eta \in X_{(1)}} \Phi(F_\eta^\h) \to \prod'_{\wp \in \BB_X} \Phi(F_\wp)$.
In particular, if $G$ is a linear algebraic group over $F$ then for $i=0,1$ there is an induced map 
$\prod_{\eta \in X_{(1)}} H^i(F_\eta^\h,G) \to \prod'_{\wp \in \BB_X} H^i(F_\wp,G)$; and similarly for $i>1$ if $G$ is commutative.

We now obtain the desired Mayer-Vietoris sequence with respect to points on the closed fiber:

\begin{thm} \label{6-term_sequence_points}
Let $F$ be a semi-global field over a complete discrete valuation ring $T$ with fraction field $K$, and let $\XX$ be a normal model of $F$ over $T$, with closed fiber $X$. 
Then for any linear algebraic group $G$ over $F$, 
we have an exact sequence of pointed sets
\begin{equation} \label{ses diagram}
\small{
\xymatrix{
1 \ar[r] & H^0(F,G) \ar[r] &
\prod_{P \in X_{(0)}} H^0(F_P,G) \times \prod_{\eta \in X_{(1)}} H^0(F_\eta^\h,G) \ar[r] & \prod'_{\wp \in \BB_X} H^0(F_\wp,G) \ar@<-2pt> `d[l]
`[lld] [lld] \\  
& H^1(F, G) \ar[r] & \prod_{P \in X_{(0)}} H^1(F_P,G) \times \prod_{\eta \in X_{(1)}} H^1(F_\eta^\h,G) \ar@<.5ex>[r] \ar@<-.5ex>[r] & \prod'_{\wp \in \BB_X} H^1(F_\wp,G).
\\
}}
\end{equation}   
\end{thm}

\begin{proof}
Let $\PP$ be a non-empty finite set of closed points of $\XX$ that includes all the points at which the closed fiber $X$ is not unibranched.  Let $\UU$ be the set of connected components of the complement of $\PP$ in $X$, and let $\BB$ be the set of branches of $X$ at the points of $\PP$.  
As discussed in Section~\ref{sect: background}, by~\cite[Theorem~3.5]{HHK:H1}, there is the exact sequence (\ref{ses diagram patch}) of pointed sets.

For $i=0,1$, let $A^i=\prod_{P \in \PP} H^i(F_P,G) \times \prod_{U \in \UU} H^i(F_U,G)$; 
let $B^i = \prod_{\wp \in \BB} H^i(F_\wp,G)$; and let $D^i=H^i(F,G)$.  Also, let 
$\PP'$ be the complement of $\PP$ in $X_{(0)}$, and let 
$C^i = \prod_{P \in \PP'} H^i(F_P,G)$.  Define $\epsilon^i:A^i \to C^i$ as the composition 
$A^i \to \prod_{U \in \UU} H^i(F_U,G) \to C^i$, where the first map is the projection onto the second factor in the definition of $A^i$, and the second map is induced by the inclusions $F_U \hookrightarrow F_P$ for each closed point $P \in U$.  (Note that each $P \in \PP'$ lies in a unique $U \in \UU$.)

For short, write $H^i_F := H^i(F,G)$ for $i=1,2$; and for any indexed set $\{F_s\}_{s \in S}$ of overfields of $F$ 
write $H^i_S :=  \prod_{s \in S} H^i(F_s,G)$ for $i=0,1$.  Applying Lemma~\ref{expand ex seq}(\ref{enlarged seq sets}) to the exact sequence (\ref{ses diagram patch}), we obtain the following expanded exact sequence:
\[\xymatrix{
1  \ar[r] & H^0_F  \ar[r]^-{\alpha_{0,\UU}} & H^0_{X_{(0)}} \times H^0_{\UU} \ar[r]^{\beta_{0,\UU}} & H^0_{\BB}
 \times H^0_{\PP'}  \ar[r]^-{\delta_{0,\UU}} & H^1_F \ar[r]^-{\alpha_{1,\UU}} & H^1_{X_{(0)}} \times H^1_\UU 
 \ar@<.5ex>[r]^{\beta_{1,\UU}'^1} \ar@<-.5ex>[r]_{\beta_{2,\UU}'^1} & H^1_\BB \times H^1_{\PP'}.  \\
}\]
Here the initial $1$, corresponding to the triviality of $\ker(\alpha_{0,\UU})$, follows from the injectivity of $H^0(F,G) \to H^0(F_P,G)$ for any $P \in X_{(0)}$.  Also, in the terms $H^i_{X_{(0)}} \times H^i_{\UU}$,
we use that $H^i_{X_{(0)}} = H^i_\PP \times H^i_{\PP'}$.

We can now let the above finite set $\PP$ vary, with $\UU$, $\BB$, and $\PP'$ varying along with it.  The sets $\PP$ form a direct system; 
and because of the compatibility of the maps in the above exact sequence, as $\PP$ varies, we obtain a direct system of exact sequences.  Here 
\[\lim_\to H^i_\UU =\!\prod_{\eta \in X_{(1)}} \!H^i(F_\eta^\h,G) \ \ \ {\rm and} \ \ \ \lim_\to\, (H^i_{\BB} \times H^i_{\PP'}) = \!{\prod_{\wp \in \BB_X}}\!\!\!'\ H^i(F_\wp,G).\]
Since direct limits preserve exactness, we obtain the asserted exact sequence of pointed sets.
\end{proof}

\begin{rem} \label{6-term_pts_rk}
Theorem~\ref{6-term_sequence_points} would no longer hold if $F_\eta^\h$ were replaced by $F_\eta$.  First note that such a replacement would also require replacing $\prod' H^i(F_\wp,G)$ by $\prod H^i(F_\wp,G)$, or else the last map in each row would no longer be defined.   Doing both replacements would lead to the first row not being exact at the middle $H^0$ term even for the group~$G=\mbb G_a$ and $\XX=\P^1_{k[[t]]}$ with $\cha(k)=0$, by Example~\ref{int counterex}.  Moreover the second row would not be exact at its middle term with the same choice of $\XX$ and taking $G$ to be cyclic of order two.  Namely, for each positive integer $n$ let $P_n$ be the point on the closed fiber $X=\P^1_k$ where $x=n$, and take the $G$-torsor over $F_{P_n}$ given by the $G$-Galois extension defined by 
$y^2=(x-n)(x-n-t)$.  At all other points on $X$ (including $\eta$), take the trivial $G$-torsor.  Then this element of the middle term of the second row is in the equalizer of the two maps to $\prod H^1(F_\wp,G)$ but is not in the image of $H^1(F,G)$, since a branched cover of $\XX$ cannot have infinitely many components of its branch locus.  (Note, though, that this element is not in the equalizer to $\prod' H^1(F_\wp,G)$, since that restricted product is defined as a direct limit, and at no finite level is the element induced by the above $G$-torsors trivial.  This also shows that $\prod' H^1(F_\wp,G)\to\prod H^1(F_\wp,G)$ is not injective.)  The same exactness issue on $H^1$ would arise if we retained $F_\eta^\h$ but
still replaced $\prod' H^i(F_\wp,G)$ by $\prod H^i(F_\wp,G)$.
\end{rem}

By analogy with the notion of a class set (or group) associated to a linear algebraic group over a global field being given by the adelic double coset space $G(F) \backslash G(\A(F,S)) / G(\A^\infty(F,S))$, we write \[\Cl_X(F,G) = \prod_{\eta \in X_{(1)}} G(F_\eta^\h) \ \backslash {\prod_{\wp \in \BB_X}}\!\!' \ G(F_\wp)\,/\!\!\prod_{P \in X_{(0)}}\!\!G(F_P)\] 
in the situation of Theorem~\ref{6-term_sequence_points}.

\begin{cor} \label{double coset H1}
The coboundary map of the exact sequence in Theorem~\ref{6-term_sequence_points} induces a bijection of pointed sets
$\Cl_X(F,G) \to \Sha_X(F,G)$.
\end{cor}

\begin{proof}
By Theorem~\ref{6-term_sequence_points}, we have a bijection from the above double coset space to the kernel of the map $H^1(F, G) \to \prod_{P \in X_{(0)}} H^1(F_P,G) \times \prod_{\eta \in X_{(1)}} H^1(F_\eta^\h,G)$.
But by \cite[Proposition~5.8]{HHK:H1} (which relied on the Artin Approximation theorem, \cite[Theorem~1.10]{artin}), if a $G$-torsor over $F$ is trivial over $F_\eta$ for some $\eta \in X_{(1)}$, then it is trivial over some $F_U\subset F_\eta$ and hence over $F_\eta^\h$.  Moreover the converse also holds, since $F_\eta^\h \subseteq F_\eta$.  Hence 
\[\ker\biggl(H^1(F, G) \to\!\prod_{P \in X_{(0)}}\!H^1(F_P,G) \times\!\prod_{\eta \in X_{(1)}}\!H^1(F_\eta^\h,G)\biggr) =\Sha_X(F,G),\]
yielding the corollary.
\end{proof}

\begin{ex} \label{Sigma ex}
\renewcommand{\theenumi}{\alph{enumi}}
\renewcommand{\labelenumi}{(\alph{enumi})}
\begin{enumerate}
\item  \label{rational conn gps}
If $G$ is a rational connected linear algebraic group, then $\Sha_X(F,G)$ is trivial, by \cite[Theorem~5.10]{HHK:H1}.  Hence $\Cl_X(F,G)$ is trivial, by Corollary~\ref{double coset H1}.  In particular, this holds for the groups $\GL_n$ and $\PGL_n$.  The triviality of $\Cl_X(F,\GL_n)$ can also be seen from the exact sequence (\ref{ses diagram}) because $H^1(F,\GL_n)$ vanishes by Hilbert~90.  The case of $\PGL_n$ can then be deduced directly from the case of $\GL_n$.  Namely, every element of $\Cl_X(F,\PGL_n)$ is represented by an element of $\prod'_{\wp \in \BB_X} \PGL_n(F_\wp)$, and this can be lifted to an element of $\prod'_{\wp \in \BB_X} \GL_n(F_\wp)$, representing an element of $\Cl_X(F,\GL_n)$.  Since $\Cl_X(F,\GL_n)$ is trivial, that class is also represented by the trivial element of $\prod'_{\wp \in \BB_X} \GL_n(F_\wp)$, and hence the given class in $\Cl_X(F,\PGL_n)$ is represented by the trivial element of $\prod'_{\wp \in \BB_X} \PGL_n(F_\wp)$.

\item \label{rational disconn gps}
If $G$ is a linear algebraic group over $F$ such that each connected component is a rational $F$-variety, then $\Sha_X(F,G)$ need not be trivial, but it is finite.  More explicitly, it is in bijection with $\Hom(\pi_1(\Gamma),G/G^0)/\!\!\sim$, where $\Gamma$ is the reduction graph of $\ms  X$, which encodes the intersections of the irreducible components of the closed fiber $X$ (see \cite[Section~6]{HHK:H1}), and where $\sim$ denotes post-conjugation by elements of the finite group $G/G^0$.  (Here $G^0$ is the identity component of $G$.)  Thus in this situation, $\Cl_X(F,G)$ is finite, and its order is computed by the above expression.

\item \label{G2 gps}
If $\cha(F) \ne 2$ and $G$ is a linear algebraic group of type $G_2$, then $\Sha_X(F,G)$ is trivial by \cite[Example~9.4]{HHK:H1} (corresponding to a local-global principle for octonion algebras).  Hence $\Cl_X(F,G)$ is trivial for such groups.

\item \label{BDE}
For each of the following types of linear algebraic groups $G$ over $F$, $\Sha_X(F,G)$ is trivial by \cite[Corollaries~4.3.2, 4.3.3]{HHK:Hi}, and hence so is  $\Cl_X(F,G)$: a quasi-split group of type $E_6$ or $E_7$; an almost simple group that is quasi-split of absolute rank at most $5$; an almost simple group that is quasi-split of type $B_6$ or $D_6$; an almost simple group that is split of type $D_7$; $\SL_1(A)$ where $A$ is a central simple $F$-algebra whose degree is square free and not divisible by $\cha(k)$.

\item \label{monotonic}
Suppose that $T=k[[t]]$ for some field $k$ of characteristic zero, and that $\XX$ is a normal crossings model whose closed fiber $X$ is reduced.  Suppose also that $G$ is defined as a linear algebraic group over $k$.  If the reduction graph associated to the closed fiber is a tree, and if this property is preserved upon finite extension of $k$, then $\Sha_X(F,G)$ is trivial by \cite[Theorem~4.11]{CHHKPS}.  Hence $\Cl_X(F,G)$ is also trivial.

\item \label{nc red cd}
If $\XX$ is a normal crossings model, and $G$ is a reductive group over $T$, then $\Sha_X(F,G) = \Sha(F,G) := \ker(H^1(F,G) \to \prod_v H^1(F_v,G))$, where $v$ ranges over the discrete valuations on $F$; see \cite[Theorem~3.2]{CHHKPS}.  In particular, this holds if $G$ is the semisimple group $\SL_1(D)$, where $D$ is a biquaternion division algebra over $k$, and $T=k[[t]]$.  Assume in addition that the closed fiber $X$ of $\XX$ consists of copies of $\P^1_k$ meeting at $k$-points.  If $\cd_2(k) \le 3$, then $\Sha(F,G)$ is trivial by \cite[Proposition~7.8]{CHHKPS}, and hence so is $\Cl_X(F,G) =\Sha_X(F,G)$.  
On the other hand, suppose we take $k=\QQ(\sqrt{17})((x))((y))$, of cohomological dimension equal to four.  Let $D=(-1,x)\otimes(2,y)$, and take $\XX = \operatorname{Proj}(T[u,v,w]/(uvw-t(u+v+w)^3))$.  Then $\Sha(F,G) = \Z/2\Z$ by \cite[Example 7.6]{CHHKPS}, and so $\Cl_X(F,G) = \Z/2\Z$.
Moreover, if we let $\tilde k/k$ be a suitable field extension (of infinite transcendence degree) and let $\tilde F$ be the base change of $F$ from $k((t))$ to $\tilde k((t))$, then $\Sha(\tilde F,G)$ is infinite by that same example.  Equivalently, $\Cl_X(\tilde F,G) = \Sha_X(\tilde F,G)$ is infinite, in contrast to the situation for global fields, where the class group is always finite (\cite[Theorem~5.1]{Borel}).  
\end{enumerate}
\end{ex}

In the case that the group $G$ is commutative, the higher cohomology groups $H^i$ are defined, for all $i \ge 0$.  In this situation, under a characteristic hypothesis, there is an associated long exact Mayer-Vietoris cohomology sequence that extends the one in 
\cite[Theorem~3.5]{HHK:H1}; see \cite[Theorem~3.1.3]{HHK:Hi}.  Using this, we obtain the following exact sequence of abelian groups (since $H^i$ with commutative coefficients is an abelian group):

\begin{thm} \label{les points}
Let $F$ be a semi-global field over a complete discrete valuation ring $T$ with fraction field $K$ and residue field $k$, and let $\XX$ be a normal model of $F$ over $T$, with closed fiber $X$.
Let $G$ be a commutative linear algebraic group over $F$, and assume that either $G$ has finite order not divisible by $\cha(k)$ or that $\cha(k)=0$.  Then there is the following long exact sequence of abelian groups:
\begin{equation} \label{les diagram}
\small{
\xymatrix{
1 \ar[r] & H^0(F,G) \ar[r] &
\prod_{P \in X_{(0)}} H^0(F_P,G) \times \prod_{\eta \in X_{(1)}} H^0(F_\eta^\h,G) \ar[r] & \prod'_{\wp \in \BB_X} H^0(F_\wp,G) \ar@<-2pt> `d[l]
`[lld] [lld] \\  
& H^1(F, G) \ar[r] & \prod_{P \in X_{(0)}} H^1(F_P,G) \times \prod_{\eta \in X_{(1)}} H^1(F_\eta^\h,G) \ar[r] & \prod'_{\wp \in \BB_X} H^1(F_\wp,G)\ar@<-2pt> `d[l]
`[lld] [lld] \\
& H^2(F, G) \ar[r] & {\ \ \cdots \hskip 2.5in}
\\
}}
\end{equation}  
\end{thm}

\begin{proof}
We proceed analogously to the proof of Theorem~\ref{6-term_sequence_points}.  Let $\PP, \UU, \BB$, $\PP'$ be as in that proof.  By \cite[Theorem~3.1.2]{HHK:Hi}, we have a long exact sequence (\ref{les diagram patch}) of abelian groups.

Define $A^i,B^i,C^i,D^i$ and $\epsilon^i:A^i \to C^i$ as in the proof of Theorem~\ref{6-term_sequence_points}, but now considering all $i\ge 0$ (since $G$ is commutative).  For each $j \ge 0$, by applying Lemma~\ref{expand ex seq}(\ref{enlarged seq groups}) to the sequence 
\[D^j \to  A^j  \to  B^j \to \ D^{j+1} \to A^{j+1}  \to  B^{j+1},\]
we obtain an exact sequence of abelian groups 
\[D^j \to  A^j \times C^j \to  B^j  \times C^j \to \ D^{j+1} \to A^{j+1}  \times C^{j+1} \to  B^{j+1} \times C^{j+1}.\]
Splicing these together yields a long exact sequence 
\[\cdots D^j \to A^j \times C^j \to  B^j  \times C^j\ \to D^{j+1} \cdots,\]
which begins with $1 \to D^0$ since $D^0 \to A^0$ is injective.  Taking direct limits as before then yields the asserted long exact sequence of abelian groups.
\end{proof}

In the above situation, for $i \ge 0$ and $G$ commutative, write 
\[\Sha_X^i(F,G) := \ker\biggl(H^i(F,G) \to \prod_P H^i(F_P,G)\biggr),\]
where $P$ ranges over all the points of $X$ (including generic points),
and
\[\Cl^i_X(F,G) :=\!\!\!\prod_{\eta \in X_{(1)}}\!H^i(F_\eta^\h,G) \ \backslash\!{\prod_{\wp \in \BB_X}}\!\!' \ H^i(F_\wp,G) \ /\!\!\prod_{P \in X_{(0)}} H^i(F_P,G).\]
Thus $\Sha_X^0(F,G)=1$, $\Sha_X^1(F,G) = \Sha_X(F,G)$, and $\Cl^0_X(F,G) = \Cl_X(F,G)$.

\begin{cor} \label{double coset Hi}
For every $i \ge 0$, the coboundary map of the above exact sequence induces a group isomorphism
$\Cl^i_X(F,G) \to \Sha_X^{i+1}(F,G)$.
\end{cor}

\begin{proof}
This follows from Theorem~\ref{les points} in the same way that Corollary~\ref{double coset H1} followed from Theorem~\ref{6-term_sequence_points}, except that \cite[Proposition~3.2.2]{HHK:Hi} is used in place of \cite[Proposition~5.8]{HHK:H1} to pass from $F_\eta$ to $F^\h_\eta$.
\end{proof}

Under appropriate hypotheses, the coboundary maps in the long exact sequence in \cite[Theorem~3.1.3]{HHK:Hi} are trivial, and so the sequence splits, yielding the triviality of $\Cl^i_X(F,G)$.  Namely, we have the examples given in the following two corollaries, where $\Z/m\Z(r) := \mu_m^{\otimes r}$ for $m$ not divisible by the residue characteristic:

\begin{cor} \label{twisted cyclic sequence}
With $F$ as in Theorem~\ref{les points} and $i \ge 1$, let $G = \Z/m\Z(r)$ for some integers $m,r$ with $m>0$ not divisible by $\cha(k)$, such that either $r=i$ or $[F(\mu_m):F]$ is prime to $m$.
Then the coboundary map $\prod'_{\wp \in \BB_X} H^i(F_\wp,G) \to H^{i+1}(F,G)$ in Theorem~\ref{les points} is trivial and $\Cl^i_X(F,G)$ is trivial.
\end{cor}

\begin{proof}
Under the above hypotheses, the map $H^{i+1}(F,G) \to \prod_{P \in X} H^{i+1}(F_P,G)$ is injective by \cite[Theorem~3.2.3(i)]{HHK:Hi}.  That is, $\Sha_X^{i+1}(F,G)$ is trivial; and then so is $\Cl_X^i(F,G)$ by Corollary~\ref{double coset Hi}.  The above injection factors through
\[\alpha_{i+1}:H^{i+1}(F, G) \to \prod_{P \in X_{(0)}} H^{i+1}(F_P,G) \times \prod_{\eta \in X_{(1)}} H^{i+1}(F_\eta^\h,G),\] 
since $H^{i+1}(F, G) \to H^{i+1}(F_\eta,G)$ factors through $H^{i+1}(F,G) \to H^{i+1}(F_\eta^\h,G)$ for each $\eta \in X_{(1)}$.  Thus 
$\alpha_{i+1}$ is also injective; and so the coboundary map $\prod'_{\wp \in \BB_X} H^i(F_\wp,G) \to H^{i+1}(F,G)$ is trivial, by Theorem~\ref{les points}.
\end{proof}

\begin{cor} \label{Gm sequence}
With $F$ as in Theorem~\ref{les points}, let $G = \Gm$.  Then the coboundary maps $\prod'_{\wp \in \BB_X} H^i(F_\wp,G) \to H^{i+1}(F,G)$ in Theorem~\ref{les points} are trivial for $i=0,1$, and $\Cl^i_X(F,G)$ is trivial.  The same holds for all $i>1$ provided that $\cha(k)=0$ and $K$ contains a primitive $m$-th root of unity for all $m \ge 1$.
\end{cor}

\begin{proof}
In the case $i=0$, Example~\ref{Sigma ex}(\ref{rational conn gps}) 
showed that $\Cl_X(F,\Gm) = \Cl^0_X(F,\Gm)$ is trivial.  Hence so is the associated coboundary map, by the exactness of (\ref{les diagram}).

For the case $i=1$, first note that $\Cl_X(F,\PGL_n)$ is also trivial by Example~\ref{Sigma ex}(\ref{rational conn gps}).  Thus by Theorem~\ref{6-term_sequence_points} in the case $G=\PGL_n$, the coboundary map is trivial and so the map \[H^1(F,\PGL_n) \to \prod_{P \in X_{(0)}} H^1(F_P,\PGL_n) \times\!\prod_{\eta \in X_{(1)}} H^1(F_\eta^\h,\PGL_n)\]
has trivial kernel.  But for any field $E$, the pointed set $H^1(E,\PGL_n)$ classifies isomorphism classes of central simple $E$-algebras of degree $n$, with the distinguished element corresponding to the split algebra ${\rm M}_n(E)$.  Thus the map $\Br(F) \to \prod_{P \in X_{(0)}} \Br(F_P) \times \prod_{\eta \in X_{(1)}} \Br(F_\eta^\h)$ has trivial kernel.  But this is just the map from $H^2(F,\Gm)$ in the exact sequence (\ref{les diagram}) with $G=\Gm$.  Hence the coboundary map to $H^2(F,\Gm)$ is trivial, and thus so is $\Cl^1_X(F,\Gm)$.

Finally, for the case $i>1$, the map $H^{i+1}(F,\Gm) \to \prod_{P \in X} H^{i+1}(F_P,\Gm)$ is injective by \cite[Theorem~3.2.3(ii)]{HHK:Hi}, under the given additional hypotheses.  The remainder of the proof of the corollary is then the same as in the proof of Corollary~\ref{twisted cyclic sequence}.
\end{proof}

\section{Adelic double cosets for relative projective curves} \label{sect: ring Cl}

The previous section considered a linear algebraic group $G$ defined over a semi-global field $F$, and worked with the Galois cohomology of $G$ over extension fields of $F$ that arose from looking locally on a model $\XX$ of $F$.  Below we consider an affine group scheme over $\XX$, and work with \'etale cohomology with respect to completions of rings associated to the model.  Again we obtain a Mayer-Vietoris sequence, which gives rise to an identification of a local-global obstruction to an adelic double coset space (now defined in terms of rings).
This gives rise to a closer analog of the classical adelic double coset space, classifying locally trivial torsors over $\XX$ rather than over $F$.  Whereas Section~\ref{sect: sgf Cl} relied on results about patching over fields (from \cite{HHK:H1} and \cite{HHK:Hi}), here we draw on patching results for torsors over schemes that were proven from a Tannakian point of view in \cite{HKL}.

To emphasize the parallel with the classical case, we begin with the ``toy model'' of a connected (but possibly reducible) one-dimensional scheme $X$, which we can think of, in the projective case, as the closed fiber of a model of a semi-global field.  The scheme $X$, however, is not required to be a projective curve, and in fact it can be the spectrum of the ring of integers of a number field.  In that situation we recover the classical case as considered by Borel (\cite{Borel}).

Below, all our affine group schemes are required to be smooth and of finite type.

As indicated above, we first consider a connected reduced one-dimensional Noetherian scheme $X$, and let $X_{(0)}, X_{(1)}$ denote the set of points of dimension $0,1$ respectively.  For each $P \in X_{(0)}$, let $\wh \OO_P$ be the completion of the local ring $\OO_P$ of $X$ at $P$; this is a complete discrete valuation ring if $X$ is regular at $P$.  For $P \in X_{(0)}$, let
the {\em branches} at $P$ be the minimal primes $\wp$ in $\wh \OO_P$.  For each branch $\wp$, let $\KK_\wp$ be the localization of $\wh \OO_P$ at $\wp$; this is the fraction field of $\wh \OO_P$ if $X$ is regular at $P$.  
If $S \subseteq X$, let $\BB_S$ be the set of branches at closed points in~$S$; in particular, $\BB_X$ is 
the set of all branches at points of $X_{(0)}$.  

If $\eta \in X_{(1)}$, then the closure of $\eta$ is an irreducible component $X_0$ of $X$; and we let $\KK_\eta$ be the function field of $X_0$.  Thus $\KK_\eta$ is the residue field at $\eta$, and it is the union of the coordinate rings $\OO(U)$, where $U$ ranges over the nonempty affine open subsets of $X_0$.
For $i=0,1$, consider the restricted product $\displaystyle {\prod_{\wp \in \BB_X}}\!\!' \,H^i(\KK_\wp,G): = \lim_\to\, \bigl(\!\prod_{\wp \in \BB_S} \!H^i(\KK_\wp,G) \times\!\prod_{\wp \not\in S} \!H^i(\wh \OO_P,G)\bigr)$, where $S$ ranges over finite subsets of $X_{(0)}$.

\begin{thm} \label{6 term 1 dim}
Let $X$ be a connected reduced one-dimensional Noetherian scheme, and let $G$ be a smooth affine group scheme over $X$.  Then there is an exact sequence in \'etale cohomology
\[\xymatrix{
1 \ar[r] & H^0(X,G) \ar[r] &
\prod_{P \in X_{(0)}} H^0(\wh \OO_P,G) \times \prod_{\eta \in X_{(1)}} H^0(\KK_\eta,G) \ar[r] & \prod'_{\wp \in \BB_X} H^0(\KK_\wp,G) \ar@<-2pt> `d[l]
`[lld] [lld] \\  
& H^1(X, G) \ar[r] & \prod_{P \in X_{(0)}} H^1(\wh \OO_P,G) \times \prod_{\eta \in X_{(1)}} H^1(\KK_\eta,G) \ar@<.5ex>[r] \ar@<-.5ex>[r] & \prod'_{\wp \in \BB_X} H^1(\KK_\wp,G).
}\]  
\end{thm}

This follows from the following analogous assertion with respect to a finite set $\PP$ of closed points, via Lemma~\ref{expand ex seq}(\ref{enlarged seq sets}), in the same way that the exact sequence (\ref{ses diagram}) in Theorem~\ref{6-term_sequence_points} followed from the exact sequence (\ref{ses diagram patch}).

\begin{prop} \label{6 term 1 dim patch}
Let $X$ be a connected reduced one-dimensional excellent scheme, and let $G$ be a smooth affine group scheme over $X$.  
Let $\PP$ be a nonempty subset of $X_{(0)}$ that contains all the points at which $X$ is not regular.  Let $\UU$ be the set of connected components of the complement $X \smallsetminus \PP$; each of these is affine.  Let $\BB$ be the set of branches of $X$ at the points of $\PP$.  Then there is a functorial exact sequence in \'etale cohomology 
\[\xymatrix{
1 \ar[r] & H^0(X,G) \ar[r] &
\prod_{P \in \PP} H^0(\wh \OO_P,G) \times \prod_{U \in \UU} H^0(\OO(U),G) \ar[r] & \prod_{\wp \in \BB} H^0(\KK_\wp,G) \ar@<-2pt> `d[l]
`[lld] [lld] \\  
& H^1(X, G) \ar[r] & \prod_{P \in \PP} H^1(\wh \OO_P,G) \times \prod_{U \in \UU} H^1(\OO(U),G) \ar@<.5ex>[r] \ar@<-.5ex>[r] & \prod_{\wp \in \BB} H^1(\KK_\wp,G).
}\]  
\end{prop}

Before proving this proposition, we first introduce some notation. For any scheme $X$, let $\fM(X)$ be the category of coherent sheaves of $\ms O_X$-modules.  For any ring $R$, let $\fM(R)=\fM(\Spec(R))$, which may be identified with the category of finitely presented $R$-modules.  Given functors $\phi_1:\ms C_1 \to \ms C_0$ and $\phi_2:\ms C_2 \to \ms C_0$, we write $\ms C_1 \times_{\ms C_0} \ms C_2$ for the 2-fiber product of $\ms C_1$ and $\ms C_2$ over $\ms C_0$.  An object in this category consists of a triple $(A_1,A_2,\psi)$ with $A_i$ an object in $\ms C_i$ and $\psi:\phi_1(A_1) \to \phi_2(A_2)$ an isomorphism; and a morphism in this category consists of a pair of compatible morphisms in $\ms C_1$ and $\ms C_2$.  Given morphisms of schemes $X_0 \to X_i \to X$ for $i=1,2$, we say that {\em patching holds for coherent sheaves} on $X$ with respect to $X_1,X_2,X_0$ if the base change functor 
$\fM(X) \to \fM(X_1) \times_{\fM(X_0)} \fM(X_2)$ is an equivalence of categories.  This says that to give a coherent sheaf on $X$ is the same as giving coherent sheaves on $X_1$ and on $X_2$ together with an isomorphism between the sheaves they induce on $X_0$.  For example, this holds if $X_1,X_2$ are Zariski open subsets of $X$ whose union is $X$ and whose intersection is $X_0$.

\begin{proof}[Proof of Proposition~\ref{6 term 1 dim patch} (and hence of Theorem~\ref{6 term 1 dim})]
Let $\wh{\PP} = \Spec(\prod_{P \in \PP} \wh {\ms O}_P)$ and let 
\[\ms P^\circ = \wh{\PP} \times_X (X \smallsetminus \PP) = \Spec(\prod_{\wp \in \BB} \ms K_\wp).\]  Here $X \smallsetminus \PP = \bigcup_{U \in \UU} U$.  Then patching holds for coherent sheaves on $X$ with respect to $\wh{\PP}, X \smallsetminus \PP, \ms P^\circ$.  In the case that $X$ is affine, this follows by \cite[Proposition~4.2]{FR} or \cite[Theorem~2.6]{artinalg}; and the case of a general $X$ is then a consequence, as shown in \cite[Theorem~3.2]{Pr}.  By \cite[Corollary~3.0.2]{HKL}, since patching holds for coherent sheaves, it follows that we have the exact sequence asserted in the proposition.
\end{proof}

\begin{cor} \label{double coset H1 curve}
The coboundary map of the exact sequence in Theorem~\ref{6 term 1 dim} induces a bijection of pointed sets
\[\Cl(X,G) \to \Sha(X,G),\]
where $\displaystyle \Cl(X,G) = \prod_{\eta \in X_{(1)}} G(\KK_\eta) \ \backslash {\prod_{\wp \in \BB_X}}\!\!' \ G(\KK_\wp) \ / \prod_{P \in X_{(0)}} G(\wh \OO_P)$ 
and where $\Sha(X,G)$ is the set of isomorphism classes of $G$-torsors over $X$ that are trivial over each complete local ring $\wh \OO_P$ {\sl and} over each generic point of $X$.
\end{cor}

\begin{ex} \label{1-dim ex}
In the situation that $X$ is a smooth connected curve over a finite field having function field $K$, or a dense open subset of $\Spec(\ms O_K)$ for some number field $K$, we recapture the classical situation.  Namely, 
if we let $S$ be the (finite) set of places of $K$ that do not correspond to (closed) points of $X$, then
$\prod'_{\wp \in \BB_X} G(\KK_\wp)$ is the same as the ad\`ele group $\A(G(K,S)) = \{(g_v)_v \in 
\prod_{v \not\in S} G(K_v)\ |\ g_v \in G(O_v) \ {\rm for\ almost\ all} \ v\}$, and 
$\prod_{P \in X_{(0)}} G(\wh \OO_P)$ equals $\A^\infty(G(K,S))$, the group of integral ad\`eles with respect to the 
places of $K$ not in $S$.
Also, there is a unique generic point $\eta$, with $\KK_\eta = K$.  Thus $\Cl(X,G) = G(K) \backslash G(\A(K,S)) / G(\A^\infty(K,S))$.  
In particular, we have the following cases:
\renewcommand{\theenumi}{\alph{enumi}}
\renewcommand{\labelenumi}{(\alph{enumi})}
\begin{enumerate}
\item  \label{global Gm}
If $G = \Gm$
then Hilbert's Theorem 90 implies that $H^1(K,G)$ is trivial, as is each $H^1(\wh \OO_P,G)$ (since $\wh \OO_P$ is local); and thus from the exact sequence in Theorem~\ref{6 term 1 dim} we recover the classical fact that $\Cl(X,G) = H^1(X,G) = \Pic(X)$, the divisor class group of~$X$.
\item  \label{global PGL}
If $G = \PGL_n$, then $H^1(X,G)$ classifies Azumaya algebras of degree $n$ over $X$, and $H^1(K,G)$ classifies
central simple algebras of degree $n$ over $K$.  Thus $H^1(X,G) \to H^1(K,G)$ is injective, since $\Br(X) \to \Br(K)$ is injective by \cite[Chapter IV, Corollary~2.6]{Milne}.
Consequently, $\Sha(X,G)$ is trivial, and hence so is $\Cl(X,G)$, by Corollary~\ref{double coset H1 curve}. 
\item \label{sssc}
In the equal characteristic case (where $X$ is a smooth curve over a finite field $k$), if $G$ is semisimple and simply connected over $k$, then $H^1(K,G)$ is trivial by \cite[Satz~A]{Harder75}; and so is $H^1(\wh\OO_{X,P},G)$, for $P$ a closed point of $X$ (e.g., see \cite[Lemma~4.4]{BeDh09}).  So {\it every} $G$-torsor over $X$ corresponds to an element of $\Sha(X,G)$ as defined above.  We thus recover the classical fact that $\Cl(X,G) = G(K) \backslash G(\A(K,S)) / G(\A^\infty(K,S))$ classifies $G$-torsors over $X$ in this situation.
\end{enumerate}
\end{ex}

We next turn to the case of a smooth affine group scheme $G$ over a model $\XX$ of a semi-global field, obtaining an associated Mayer-Vietoris sequence in \'etale cohomology and a description of the associated double coset space as a local-global obstruction, paralleling the above case of a one-dimensional scheme.  
As in that case, we rely on \cite{HKL}.

\begin{thm} \label{6-term_sequence_points rings}
Let $T$ be a complete discrete valuation ring, let
$\XX$ be an integral flat projective $T$-curve with closed fiber $X$, and let $G$ be a smooth affine group scheme over $\XX$. 
Then there is an exact sequence in \'etale cohomology
\[\xymatrix{
1 \ar[r] & H^0(\XX,G) \ar[r] &
\prod_{P \in X_{(0)}} H^0(\wh R_P,G) \times \prod_{\eta \in X_{(1)}} H^0(R_\eta^\h,G) \ar[r] & \prod'_{\wp \in \BB_X} H^0(\wh R_\wp,G) \ar@<-2pt> `d[l]
`[lld] [lld] \\  
& H^1(\XX, G) \ar[r] & \prod_{P \in X_{(0)}} H^1(\wh R_P,G) \times \prod_{\eta \in X_{(1)}} H^1(R_\eta^\h,G) \ar@<.5ex>[r] \ar@<-.5ex>[r] & \prod'_{\wp \in \BB_X} H^1(\wh R_\wp,G).
}\]  
\end{thm}

As in the case of Theorem~\ref{6 term 1 dim}, this follows via Lemma~\ref{expand ex seq}(\ref{enlarged seq sets}) from the analogous assertion with respect to a finite set, viz.:

\begin{prop}
Under the assumptions of Theorem~\ref{6-term_sequence_points rings}, let $\PP$ be a non-empty finite set of closed points of $\XX$ that includes all the points at which the closed fiber $X$ is not unibranched, let $\UU$ be the set of connected components of the complement of $\PP$ in $X$, and let $\BB$ be the set of branches of $X$ at the points of $\PP$.  
Then there is a functorial exact sequence of pointed sets in \'etale cohomology
\[
\xymatrix{
1 \ar[r] & H^0(\XX,G) \ar[r] &
\prod_{P \in \PP} H^0(\wh R_P,G) \times \prod_{U \in \UU} H^0(\wh R_U,G) \ar[r] &
\prod_{\wp \in \BB} H^0(\wh R_\wp,G) \ar@<-2pt> `d[l]
`[lld] [lld] \\
& H^1(\XX, G) \ar[r] & \prod_{P \in \PP} H^1(\wh R_P,G) \times \prod_{U \in \UU}
H^1(\wh R_U,G) \ar@<.5ex>[r] \ar@<-.5ex>[r] & \prod_{\wp \in \BB}
H^1(\wh R_\wp,G). 
}\] 
\end{prop} 

\begin{proof}
Proceeding as in the proof of Proposition~\ref{6 term 1 dim patch}, we let $\wh{\PP} = \Spec(\prod_{P \in \PP} \wh R_P)$, let $\wh{\UU} = \Spec(\prod_{U \in \UU} \wh R_U)$, and let $\wh{\BB} = \Spec(\prod_{\wp \in \BB} \wh R_\wp)$.  Then patching holds for coherent sheaves on $\XX$ with respect to $\wh{\PP}, \wh{\UU}, \wh{\BB}$ by \cite[Theorem~3.4]{Pr} (see also \cite[Theorem~3.1.4]{HKL}), as a consequence of \cite[Proposition~4.2]{FR} or \cite[Theorem~2.6]{artinalg} together with Grothendieck's Existence Theorem (\cite[Th\'eor\`eme~5.1.6]{EGA III}).  
By \cite[Corollary~3.0.2]{HKL}, it follows from this patching property that we have the exact sequence asserted in the proposition.
\end{proof}

\begin{cor} \label{double coset H1 rings}
The coboundary map of the exact sequence in Theorem~\ref{6-term_sequence_points rings} induces a bijection of pointed sets
\[\Cl_X(\XX,G) \to \Sha_X(\XX,G),\] 
where $\displaystyle \Cl_X(\XX,G) = \prod_{\eta \in X_{(1)}} G(R_\eta^\h) \ \backslash {\prod_{\wp \in \BB_X}}\!\!' \ G(\wh R_\wp) \ / \prod_{P \in X_{(0)}} G(\wh R_P)$ and where $\Sha_X(\XX,G)$ is the set of isomorphism classes of $G$-torsors over $\XX$ that are trivial over the complete local ring $\wh R_P = \wh \OO_{\XX,P}$ at $P$ for every point $P$ of the closed fiber $X$ of $\XX$ (including the generic points of $X$).
\end{cor}

\begin{proof}
By the exact sequence in Theorem~\ref{6-term_sequence_points rings}, it suffices to show that $\Sha_X(\XX,G)$ is the kernel of 
\[H^1(\XX, G) \to \prod_{P \in X_{(0)}} H^1(\wh R_P,G) \times \prod_{\eta \in X_{(1)}} H^1(R_\eta^\h,G).\]
For this, it suffices to show that a $G$-torsor over $\XX$ that becomes trivial over $\wh R_\eta$ must also become trivial over $R_\eta^\h$; i.e., if it has an $\wh R_\eta$-point then it has an $R_\eta^\h$-point.  Since $R_\eta^\h$ contains the henselization $\til R_\eta$ of $R_\eta$, it suffices to show that a $G$-torsor over $\XX$ with an $\wh R_\eta$-point has an $\til R_\eta$-point.
This in turn follows from the Artin Approximation Theorem (see \cite[Theorem~1.10 and Theorem~1.12]{artin}), which applies because the complete discrete valuation ring $T$ is excellent (by \cite[Scholie 7.8.3(iii)]{EGA IV}) and because $R_\eta$ is the localization of a $T$-algebra of finite type at a prime ideal. 
\end{proof}

\begin{ex}
The analogs of Example~\ref{1-dim ex}(\ref{global Gm}),(\ref{global PGL}) hold in the situation of Theorem~\ref{6-term_sequence_points rings} and Corollary~\ref{double coset H1 rings}:
\renewcommand{\theenumi}{\alph{enumi}}
\renewcommand{\labelenumi}{(\alph{enumi})}
\begin{enumerate}
\item  \label{ring proper Gm}
If $G = \Gm$
then Hilbert's Theorem 90 implies that $H^1(\wh R_P,G)$ and $H^1(R_\eta^\h,G)$ are trivial, since those rings are local, and so the exact sequence in Theorem~\ref{6-term_sequence_points rings} yields that $\Cl_X(\XX,G) = H^1(\XX,G) = \Pic(\XX)$ (thereby providing a justification for the $\Cl$ notation).
\item  \label{ring proper PGL}
If $G = \PGL_n$, then we have the commutative diagram
\[\xymatrix{
H^1(\XX, G) \ar[r] \ar[d] & \prod_{P \in X_{(0)}} H^1(\wh R_P,G) \times \prod_{\eta \in X_{(1)}}
H^1(R_\eta^\h,G) \ar[d]\\
H^1(F, G) \ar[r] & \prod_{P \in X_{(0)}} H^1(F_P,G) \times \prod_{\eta \in X_{(1)}} H^1(F_\eta^\h,G). 
}\] 
Here the bottom horizontal arrow has trivial kernel by Example~\ref{Sigma ex}(\ref{rational conn gps}), since that kernel is the same as $\Cl_X(F,\PGL_n)$ by Theorem~\ref{6-term_sequence_points}.  Also, the left vertical arrow has trivial kernel by \cite[Chapter IV, Corollary~2.6]{Milne}, as in Example~\ref{1-dim ex}(\ref{global PGL}).  So by commutativity of the diagram, the top horizontal arrow has trivial kernel; i.e., $\Sha_X(\XX,\PGL_n)$ is trivial
(yielding a local-global principle for the Brauer group in this context).  By Corollary~\ref{double coset H1 rings}, $\Cl_X(\XX,\PGL_n)$ is also trivial.
\end{enumerate}
\end{ex}


\medskip

\noindent{\bf Author Information:}\\

\noindent David Harbater\\
Department of Mathematics, University of Pennsylvania, Philadelphia, PA 19104-6395, USA\\
email: harbater@math.upenn.edu


\begin{thebibliography}{CHHKPS21}

\bibitem[Art69]{artin} 
Michael Artin.
\newblock Algebraic approximation of structures over complete local rings.
\newblock  Publ.\ Math.\ IHES, vol.~36, 23--58, 1969.

\bibitem[Art70]{artinalg} 
Michael Artin.
\newblock Algebraization of formal moduli: II. Existence of modifications. 
\newblock Annals of Math., ser.\ 2, {\bf 91} (1970), 88--135.

\bibitem[BD09]{BeDh09}
Kai Behrend and Ajneet Dhillon.
\newblock Connected components of moduli stacks of torsors via Tamagawa numbers. 
\newblock Canad.\ J.\ Math., {\bf 61} (2009), no.~1, 3--28.

\bibitem[Bor63]{Borel} 
Armand Borel.
\newblock Some finiteness properties of adele groups over number fields.
\newblock Inst.\ Hautes \'Etudes Sci.\ Publ.\ Math., No.~16 (1963), 5-30.

\bibitem[CRR20]{CRR}
Vladimir I.\ Chernousov, Andrei S.\ Rapinchuk, and Igor A.\ Rapinchuk.  
\newblock The finiteness of the genus of a finite-dimensional division algebra, and some generalizations. 
\newblock Israel J.\ Math.\ {\bf 236 }(2020), no.~2, 747--799.

\bibitem[CH15]{CH}
Jean-Louis Colliot-Th\'el\`ene and David Harari.
\newblock Dualit\'e et principe local-global pour les tores sur une courbe au-dessus de $\mbb C((t))$. 
\newblock  Proc.\ Lond.\ Math.\ Soc., (3) {\bf 110} (2015), no. 6, 1475--1516.

\bibitem[CHHKPS21]{CHHKPS}
Jean-Louis\ Colliot-Th\'el\`ene, David Harbater, Julia Hartmann, Daniel Krashen, R.~Parimala, and V.~Suresh.
\newblock Local-global principles for constant reductive groups over semi-global fields.
\newblock 2021 manuscript, available at arXiv:2108.12349.

\bibitem[Con12]{Con}
Brian Conrad.
\newblock Finiteness theorems for algebraic groups over function fields.
\newblock Compositio Math.\ {\bf 148} (2012), 555--639.

\bibitem[FR70]{FR}
Daniel Ferrand and Michel Raynaud.
\newblock Fibres formelles d'un anneau local noeth\'erien. 
\newblock Ann.\ Sci.\ \'Ecole Norm.\ Sup., S\'erie 4, {\bf 3} (1970), 295--311.

\bibitem[Gro61]{EGA III}
Alexander Grothendieck. 
{\it Elements de g\'eom\'trie alg\'brique} (EGA) III $1^e$ partie. 
\newblock Publ.\ Math.\ IHES, vol.\ 11, 1961.

\bibitem[Gro65]{EGA IV}
Alexander Grothendieck. 
{\it Elements de g\'eom\'trie alg\'brique} (EGA) IV, $2^e$ partie.
\newblock Publ.\ Math.\ IHES, vol.\ 24, 1965.

\bibitem[HKL20]{HKL}
Bastian Haase, Daniel Krashen, and Max Lieblich.
\newblock A Tannakian approach to patching.
\newblock 2020 manuscript, available at arXiv:2004.11933.

\bibitem[HSS15]{HSS}
David Harari, Claus Scheiderer, and Tam\'as Szamuely.
\newblock Weak approximation for tori over $p$-adic function fields. 
\newblock Int.\ Math.\ Res.\ Not.\ IMRN {\bf 2015}, no.~10, 2751--2783.

\bibitem[HS16]{HS16}
David Harari and Tam\'as Szamuely.
\newblock Local-global questions for tori over {$p$}-adic function fields.
\newblock J.\ Algebraic Geom., {\bf 25} (2016), no.~3, 571--605.

\bibitem[Har87]{Ha87}
David Harbater.
\newblock Galois coverings of the arithmetic line. 
In: Number theory (New York, 1984-1985), 165--195,
Lecture Notes in Math., vol.~1240, Springer, Berlin, 1987.

\bibitem[Har03]{Ha:MSRI}
David Harbater.
\newblock Patching and Galois theory. In: {\it Galois groups and fundamental groups}, pp.~313--424, Math.\ Sci.\ Res.\ Inst.\ Publ., {\bf 41}, Cambridge Univ.\ Press, Cambridge, 2003.

\bibitem[HH10]{HH:FP}
David Harbater and Julia Hartmann.
\newblock Patching over fields.
\newblock Israel J.\ Math.\ \textbf{176} (2010), 61--107.

\bibitem[HHK09]{HHK}
David Harbater, Julia Hartmann, and Daniel Krashen. 
\newblock Applications of patching to quadratic forms and central simple algebras. 
\newblock Invent.\ Math.\ {\bf 178} (2009), no.~2, 231--263.

\bibitem[HHK14]{HHK:Hi}
David Harbater, Julia Hartmann, and Daniel Krashen. 
\newblock Local-global principles for Galois cohomology.
\newblock Comment.\ Math.\ Helv.\ \textbf{89} (2014), 215--253.

\bibitem[HHK15]{HHK:H1}
David Harbater, Julia Hartmann, and Daniel Krashen.
\newblock  Local-global principles for torsors over arithmetic curves.
\newblock American Journal of Mathematics, {\bf 137}  (2015), 1559--1612.

\bibitem[Har67]{Harder}
Gunter Harder.
\newblock Halbeinfache Gruppenschemata \"uber Dedekindringen.
\newblock Invent.\ Math., {\bf 4} (1967), 165--191.

\bibitem[Har75]{Harder75}
Gunter Harder.
\newblock \"Uber die Galoiskohomologie halbeinfacher algebraischer Gruppen: III. 
\newblock J.\ Reine Angew.\ Math.\ {\bf 274-275} (1975), 125--138.

\bibitem[Liu02]{Liu}
Qing Liu.  
\newblock {\it Algebraic geometry and arithmetic curves}. 
\newblock Translated from the French by Reinie Ern\'e. Oxford Graduate Texts in Mathematics, vol.~6. Oxford University Press, Oxford, 2002.

\bibitem[Mil80]{Milne}
James S.~Milne.
\newblock {\it \'Etale cohomology}. Princeton Mathematical Series, No.~33. Princeton Univ.\ Press, Princeton, N.J., 1980.

\bibitem[PS14]{PS}
R.~Parimala and V.~Suresh.
\newblock Period-index and u-invariant questions for function fields over complete discretely valued fields. \newblock Invent.\ Math.\ {\bf 197} (2014), no.~1, 215--235.

\bibitem[PR94]{PlaRap}
Vladimir Platonov and Andrei Rapinchuk.  
\newblock {\it Algebraic groups and number theory}. 
Translated from the 1991 Russian original by Rachel Rowen. Pure and Applied Mathematics, vol.~139. Academic Press, Inc., Boston, MA, 1994.

\bibitem[Pri00]{Pr}
Rachel Pries.
\newblock Construction of covers with formal and rigid geometry. In: {\it Courbes semi-stables et groupe fondamental en g\'eom\'etrie alg\'ebrique (Luminy, 1998)}, 157--167, Progr.\ Math., vol.~187, Birkh\"auser, Basel, 2000.

\bibitem[RR21]{RR:CR}
Andrei S.\ Rapinchuk and Igor A.\ Rapinchuk.
\newblock Finiteness theorems for algebraic tori over function fields.
\newblock Comptes Rendus Math\'ematique, {\bf 359} (2021), no.~8, 939--944. 

\end{thebibliography}
\end{document}